\documentclass[twocolumn]{article}
\usepackage{dsfont,amsmath,amssymb,amsthm}

\newcommand{\lcp}{\mathrm{LCP}}
\newcommand{\R}{\mathds{R}}

\newcommand{\NP}{\mathrm{NP}}

\newcommand{\Or}{{\cal O}}

\newtheorem{definition}{Definition}[section]
\newtheorem{lem}[definition]{Lemma}

\newtheorem{thm}[definition]{Theorem}

\author{B. G\"artner and M. Sprecher\thanks{Institute of Theoretical
    Computer Science, ETH Zurich, CH-8092 Zurich, Switzerland
    (\texttt{gaertner@inf.ethz.ch})}} \title{A Polynomial-Time
  Algorithm for the Tridiagonal and Hessenberg P-Matrix Linear
  Complementarity Problem}

\begin{document}
\maketitle

\begin{abstract}
  We give a polynomial-time dynamic programming algorithm for solving
  the linear complementarity problem with tridiagonal or, more
  generally, Hessenberg P-matrices. We briefly review three known
  tractable matrix classes and show that none of them contains all
  tridiagonal P-matrices.
\end{abstract}

\section{Introduction}
Given a matrix $M\in\R^{n\times n}$ and a vector $q\in\R^n$, the
\emph{linear complementarity problem} $\lcp(M,q)$ is to find vectors
$w,z\in\R^n$ such that
\begin{equation}\label{eq:lcpdef}
w-Mz = q, \quad w,z\geq 0, \quad w^Tz = 0.
\end{equation}
It is $\NP$-complete in general to decide whether such vectors exist
\cite{Chung:Hardness}. But if $M$~is a \emph{P-matrix} (meaning that all
principal minors---determinants of principal submatrices---are
positive), then there are unique solution vectors $\tilde w,\tilde z$
for every right-hand side $q$~\cite{STW}. It is unknown whether these
vectors can be found in polynomial
time~\cite{Meg:A-Note-on-the-Complexity}.

The matrix $M=(m_{ij})_{i,j=1}^n$~is \emph{tridiagonal} if $m_{ij}=0$
for $|j-i|>1$. More generally, $M$ is \emph{lower Hessenberg} if
$m_{ij}=0$ for $j-i>1$, and $M$ is \emph{upper Hessenberg} if $M^T$ is
lower Hessenberg; see Figure~\ref{fig:matrices}.

\begin{figure}[htb]
\begin{center}
\begin{picture}(140,60)
% first matrix
% grid
\linethickness{0.5pt}
\put(0,0){\line(1,0){60}}
\put(0,10){\line(1,0){60}}
\put(0,20){\line(1,0){60}}
\put(0,30){\line(1,0){60}}
\put(0,40){\line(1,0){60}}
\put(0,50){\line(1,0){60}}
\put(0,60){\line(1,0){60}}
\put(0,0){\line(0,1){60}}
\put(10,0){\line(0,1){60}}
\put(20,0){\line(0,1){60}}
\put(30,0){\line(0,1){60}}
\put(40,0){\line(0,1){60}}
\put(50,0){\line(0,1){60}}
\put(60,0){\line(0,1){60}}
% nonzero elements
\linethickness{2pt}
% upper chain
\put(0,60){\line(1,0){20}}
\put(20,60){\line(0,-1){10}}
\put(20,50){\line(1,0){10}}
\put(30,50){\line(0,-1){10}}
\put(30,40){\line(1,0){10}}
\put(40,40){\line(0,-1){10}}
\put(40,30){\line(1,0){10}}
\put(50,30){\line(0,-1){10}}
\put(50,20){\line(1,0){10}}
\put(60,20){\line(0,-1){20}}
% lower chain
\put(0,60){\line(0,-1){20}}
\put(0,40){\line(1,0){10}}
\put(10,40){\line(0,-1){10}}
\put(10,30){\line(1,0){10}}
\put(20,30){\line(0,-1){10}}
\put(20,20){\line(1,0){10}}
\put(30,20){\line(0,-1){10}}
\put(30,10){\line(1,0){10}}
\put(40,10){\line(0,-1){10}}
\put(40,0){\line(1,0){20}}
% second matrix 
% grid
\linethickness{0.5pt}
\put(80,0){\line(1,0){60}}
\put(80,10){\line(1,0){60}}
\put(80,20){\line(1,0){60}}
\put(80,30){\line(1,0){60}}
\put(80,40){\line(1,0){60}}
\put(80,50){\line(1,0){60}}
\put(80,60){\line(1,0){60}}
\put(80,0){\line(0,1){60}}
\put(90,0){\line(0,1){60}}
\put(100,0){\line(0,1){60}}
\put(110,0){\line(0,1){60}}
\put(120,0){\line(0,1){60}}
\put(130,0){\line(0,1){60}}
\put(140,0){\line(0,1){60}}
% nonzero elements
\linethickness{2pt}
% upper chain
\put(80,60){\line(1,0){20}}
\put(100,60){\line(0,-1){10}}
\put(100,50){\line(1,0){10}}
\put(110,50){\line(0,-1){10}}
\put(110,40){\line(1,0){10}}
\put(120,40){\line(0,-1){10}}
\put(120,30){\line(1,0){10}}
\put(130,30){\line(0,-1){10}}
\put(130,20){\line(1,0){10}}
\put(140,20){\line(0,-1){20}}
% lower chain
\put(80,0){\line(1,0){60}}
\put(80,0){\line(0,1){60}}
\end{picture}
\end{center}
\caption{A tridiagonal matrix (left) and a lower Hessenberg matrix (right);
the nonzero entries are enclosed in bold lines.\label{fig:matrices}}
\end{figure}
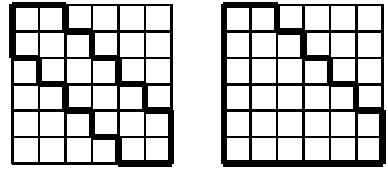

In this note we show that $\lcp(M,q)$ can be solved in polynomial time
if $M$ is a lower (or upper) Hessenberg P-matrix. Polynomial-time
results already exist for other classes of matrices, most notably
Z-matrices~\cite{Chandra}, hidden Z-matrices~\cite{Manga}, and
transposed hidden K-matrices~\cite{Pang}. Section~\ref{sec:tridiag}
shows that none of these classes contains all tridiagonal P-matrices.

For the remainder of this note, we fix a P-matrix $M\in\R^{n\times n}$
and a vector $q\in\R^{n}$.

\section{The optimal basis}

For $B\subseteq [n]:=\{1,2,\ldots,n\}$, we let $\overline{M}_B$ be the
$n\times n$ matrix whose $i$th column is the $i$th column of~$-M$ if
$i\in B$, and the $i$th column of the $n\times n$ identity
matrix~$I_n$ otherwise. $\overline{M}_B$ is invertible for every set
$B$, a direct consequence of $M$ having nonzero principal minors. We
call $B$ a \emph{basis} and $\overline{M}_B$ the associated
\emph{basis matrix}.

The \emph{complementary pair} $(w(B),z(B))$ associated with the basis 
$B$ is defined by
\begin{equation}\label{eq:lcpw}
w_i(B) := \begin{cases}
0 & \text{if $i\in B$} \\
(\overline{M}_B^{-1}q)_i & \text{if $i\notin B$}
\end{cases}
\end{equation}
and 
\begin{equation}\label{eq:lcpz} 
z_i(B) := \begin{cases}
(\overline{M}_B^{-1}q)_i & \text{if $i\in B$}\\
0 & \text{if $i\notin B$}
\end{cases},
\end{equation}
for all $i\in [n]$.

\begin{lem}\label{lem:optB} 
  For every basis $B\subseteq[n]$, the following two statements are
  equivalent. 
  \begin{itemize}
  \item[(i)] The pair $(w(B),z(B))$ solves $\lcp(M,q)$, meaning that
    $w=w(B), z=z(B)$ satisfy (\ref{eq:lcpdef}).
  \item[(ii)] $\overline{M}_B^{-1}q\geq 0$.
  \end{itemize}
\end{lem}
If both statements hold, $B$ is called an \emph{optimal} basis for
$\lcp(M,q)$.

\begin{proof}
  As a consequence of (\ref{eq:lcpw}) and (\ref{eq:lcpz}), $w=w(B)$
  and $z=z(B)$ already satisfy $w-Mz=q$ and $w^Tz=0$, for every $B$.
  Moreover,  $w,z\geq 0$ if and only if $w(B),z(B)\geq 0$;
  this in turn is equivalent to $\overline{M}_B^{-1}q\geq 0$.
\end{proof}

From now on, we assume w.l.o.g.\ that $\lcp(M,q)$ is
\emph{nondegenerate}, meaning that $(\overline{M}_B^{-1}q)_i\neq 0$
for all $B\subseteq [n]$ and all $i\in [n]$. We can achieve this e.g.\
through a symbolic perturbation of $q$. In this case, we obtain the
following

\begin{lem}\label{lem:uniqueopt} 
  There is a unique optimal basis $\tilde B$ for $\lcp(M,q)$.
\end{lem}

\begin{proof}
  Let $\tilde w,\tilde z$ be solution vectors of $\lcp(M,q)$, and set
  $\tilde B:=\{i\in [n]: \tilde w_i=0\}$. Since $\tilde w^T\tilde
  z=0$, we have $\tilde z_i=0$ if $i\in [n]\setminus \tilde B$. Hence,
  the vectors $\tilde w,\tilde z$ satisfy
  \[q=(I_n\mid -M)\left(\begin{array}{c}\tilde w\\ \tilde z
\end{array}\right)=
\overline{M}_{\tilde B}\left(\begin{array}{c}\tilde w_{[n]\setminus \tilde B}\\
    \tilde z_{\tilde B}\end{array}\right),\] so $\tilde w = w(\tilde
B),\tilde z=z(\tilde B)$ follows. Hence, $\tilde B$ satisfies
statement (i) in Lemma~\ref{lem:optB} and is therefore an optimal
basis.
 
Uniqueness of $\tilde w,\tilde z$~\cite{STW} implies via
Lemma~\ref{lem:optB} that $(w(B),z(B))=(w(\tilde B),z(\tilde B))$ for
every optimal basis $B$. But then (\ref{eq:lcpw}) and (\ref{eq:lcpz})
show that $(\overline{M}_{\tilde
  B}^{-1}q)_i=(\overline{M}_{B}^{-1}q)_i=0$ for all $i\in \tilde
B\oplus B$. Under nondegeneracy, there can be no such $i$, hence
$B=\tilde B$.
\end{proof}

\section{Subproblems}
For $K\subseteq[n]$, let $M_{KK}$ be the principal submatrix of $M$
consisting of all entries $m_{ij}$ with $i,j\in K$. Furthermore, let
$q_K$ be the subvector of $q$ consisting of all entries $q_i,i\in K$.

By definition, the submatrix $M_{KK}$ is also a P-matrix, and
$\lcp(M_{KK},q_K)$ is easily seen to inherit nondegeneracy from
$\lcp(M,q)$. Hence, Lemma~\ref{lem:uniqueopt} allows us to make the
following

\begin{definition}
  For $k\in[n]$, $B(k)\subseteq[k]$ is the unique optimal basis of
  $\lcp(M_{[k][k]},q_{[k]})$.
\end{definition}
We also set $B(-1)=B(0)=\emptyset$.

\section{The lower Hessenberg case}\label{sec:4}
Let $M$ be a lower Hessenberg matrix. Then we have the following
\begin{thm}\label{thm:dyn}
For every $k\in[n]$, there exists an index $\ell\in\{-1,0,\ldots,k-1\}$ such that
\[B(k) = B(\ell)\cup\{\ell+2,\ell+3,\ldots,k\}.\]
\end{thm}
\begin{proof}
  If $B(k)=[k]$, the statement holds with index $\ell=-1$. Otherwise,
  let $\ell\in\{0,1,\ldots,k-1\}$ be the largest index such that
  $\ell+1\notin B(k)$.  The matrix $M_{[k][k]}$ is lower Hessenberg as
  well, which implies that the basis matrix
  $\overline{M}:=\overline{M_{[k][k]}}_{B(k)}$ associated with $B(k)$
  satisfies $\overline{m}_{ij}=0$ if $i\leq\ell< j$; see
  Figure~\ref{fig:blocks}.
\begin{figure}[htb]
\begin{center}
\begin{picture}(140,60)
% first matrix
% grid
\linethickness{0.5pt}
\put(0,0){\line(1,0){60}}
\put(0,10){\line(1,0){60}}
\put(0,20){\line(1,0){60}}
\put(0,30){\line(1,0){60}}
\put(0,40){\line(1,0){60}}
\put(0,50){\line(1,0){60}}
\put(0,60){\line(1,0){60}}
\put(0,0){\line(0,1){60}}
\put(10,0){\line(0,1){60}}
\put(20,0){\line(0,1){60}}
\put(30,0){\line(0,1){60}}
\put(40,0){\line(0,1){60}}
\put(50,0){\line(0,1){60}}
\put(60,0){\line(0,1){60}}
% nonzero elements
\linethickness{2pt}
% upper chain
\put(0,60){\line(1,0){20}}
\put(20,60){\line(0,-1){10}}
%\put(20,50){\line(1,0){10}}
%\put(30,50){\line(0,-1){10}}
\put(30,40){\line(1,0){10}}
\put(40,40){\line(0,-1){10}}
\put(40,30){\line(1,0){10}}
\put(50,30){\line(0,-1){10}}
\put(50,20){\line(1,0){10}}
\put(60,20){\line(0,-1){20}}
% shortcuts
\put(20,50){\line(0,-1){10}}
\put(20,40){\line(1,0){10}}
\put(20,30){\line(1,0){10}}
\put(30,30){\line(0,-1){20}}
% lower chain
\put(0,60){\line(0,-1){20}}
\put(0,40){\line(1,0){10}}
\put(10,40){\line(0,-1){10}}
\put(10,30){\line(1,0){10}}
%\put(20,30){\line(0,-1){10}}
%\put(20,20){\line(1,0){10}}
%\put(30,20){\line(0,-1){10}}
\put(30,10){\line(1,0){10}}
\put(40,10){\line(0,-1){10}}
\put(40,0){\line(1,0){20}}
% numbers
\put(23, 2){0}
\put(23, 12){0}
\put(23, 22){0}
\put(23, 32){1}
\put(23, 42){0}
\put(23, 52){0}
% second matrix 
% grid
\linethickness{0.5pt}
\put(80,0){\line(1,0){60}}
\put(80,10){\line(1,0){60}}
\put(80,20){\line(1,0){60}}
\put(80,30){\line(1,0){60}}
\put(80,40){\line(1,0){60}}
\put(80,50){\line(1,0){60}}
\put(80,60){\line(1,0){60}}
\put(80,0){\line(0,1){60}}
\put(90,0){\line(0,1){60}}
\put(100,0){\line(0,1){60}}
\put(110,0){\line(0,1){60}}
\put(120,0){\line(0,1){60}}
\put(130,0){\line(0,1){60}}
\put(140,0){\line(0,1){60}}
% nonzero elements
\linethickness{2pt}
% upper chain
\put(80,60){\line(1,0){20}}
\put(100,60){\line(0,-1){20}}
%\put(100,50){\line(1,0){10}}
%\put(110,50){\line(0,-1){10}}
\put(100,40){\line(1,0){10}}
\put(110,40){\line(1,0){10}}
\put(120,40){\line(0,-1){10}}
\put(120,30){\line(1,0){10}}
\put(130,30){\line(0,-1){10}}
\put(130,20){\line(1,0){10}}
\put(140,20){\line(0,-1){20}}
% lower chain
\put(80,0){\line(1,0){20}}
\put(110,0){\line(1,0){30}}
\put(80,0){\line(0,1){60}}
% additional stuff
\put(100,0){\line(0,1){30}}
\put(110,0){\line(0,1){30}}
\put(100,30){\line(1,0){10}}
% numbers
\put(103, 2){0}
\put(103, 12){0}
\put(103, 22){0}
\put(103, 32){1}
\put(103, 42){0}
\put(103, 52){0}
% arrows
\put(23, -12){$\uparrow$ column $\ell+1$}
\put(103, -12){$\uparrow$ column $\ell+1$}
\end{picture}
\end{center}
\caption{The basis matrix $\overline{M}=\overline{M_{[k][k]}}_{B(k)}$
  in the tridiagonal and lower Hessenberg case if $\ell+1\notin B_k$.
  We have
  $\overline{M}_{[\ell][\ell]}=\overline{M_{[\ell][\ell]}}_{B(k)\cap[\ell]}$.
  \label{fig:blocks}}
\end{figure}
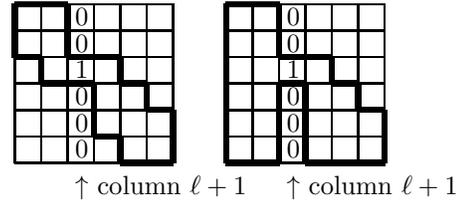

As a consequence, the system of $k$ equations
  \begin{equation}\label{eq:mainsys}
  \overline{M_{[k][k]}}_{B(k)}x=q_{[k]}
  \end{equation}
  includes the $\ell$ equations
  \begin{equation}\label{eq:subsys}
  \overline{M_{[\ell][\ell]}}_{B(k)\cap[\ell]}x_{[\ell]}=q_{[\ell]}.
\end{equation}
Since $B(k)$ is the optimal basis of $\lcp(M_{[k][k]},q_{[k]})$, the
unique solution $\tilde x$ of (\ref{eq:mainsys}) satisfies $\tilde
x\geq 0$; see Lemma~\ref{lem:optB}. Vice versa, the unique partial
solution $\tilde x_{[\ell]}\geq 0$ of subsystem (\ref{eq:subsys})
shows that $B(k)\cap[\ell]=B(\ell)$, the unique optimal basis of
$\lcp(M_{[\ell][\ell]},q_{[\ell]})$. Together with the choice of
$\ell$, the statement of the theorem follows.
\end{proof}

We remark that a variant of Theorem~\ref{thm:dyn} for upper Hessenberg
matrices can be obtained by considering \emph{lower right} principal
submatrices $M_{KK}$.

\section{Polynomial-time algorithm}
A \emph{basis test} is a procedure to decide whether a given basis
$B\subseteq[n]$ is optimal for $\lcp(M,q)$. According to
Lemma~\ref{lem:optB}, a basis test can be implemented in polynomial
time, using Gaussian elimination. In the sequel, we will therefore
adopt the number of basis tests as a measure of algorithmic
complexity. Here is our main result.

\begin{thm}\label{thm:main} 
  Let $M\in\R^{n\times n}$ be a lower Hessenberg P-matrix.
  The optimal basis $\tilde B=B(n)$ of $\lcp(M,q)$ can be found with
  at most $\binom{n+1}{2}$ basis tests.
\end{thm}
\begin{proof}
  We successively compute the optimal bases $B(-1),B(0),\ldots,B(n)$,
  where $B(-1)=B(0)=\emptyset$. To determine $B(k),k>0$, we simply
  test the $k+1$ candidates for $B(k)$ that are given by
  Theorem~\ref{thm:dyn}.  In fact, we already know $B(k)$ after
  testing $k$ of the candidates. This algorithm requires a total of
  $\sum_{k=1}^n k =\binom{n+1}{2}$ basis tests.
\end{proof}
Using an $O(n^3)$ Gaussian elimination procedure, we obtain an $O(n^5)$
algorithm---this is certainly not best possible. Faster algorithms are
available if $M$ is a tridiagonal $Z$-matrix~\cite{Gupta} or
$K$-matrix~\cite{Cryer,Cottle}, but to our knowledge, the above
algorithm is the first one to handle tridiagonal (and lower
Hessenberg) P-matrices in polynomial time. The case of upper Hessenberg
matrices is analogous, see the remark at the end of Section~\ref{sec:4}.

All upper and lower triangular P-matrices are 
hidden Z~\cite{tsatso:generating}, meaning that linear complementarity
problems with triangular P-matrices can be solved in polynomial
time~\cite{Manga}. We can now also handle the ``almost'' triangular
Hessenberg P-matrices. As we show next, there is a significant
combinatorial difference between the two classes.

\section{A tridiagonal example}\label{sec:tridiag}
Let us consider $\lcp(M,q)$ with 
\begin{equation}\label{eq:Mq}
M = \left(\begin{array}{rrrr}
36 & -81 & 0 & 0 \\
147 & 16 & -74 & 0 \\
0 & 114 & 28 & 171 \\
0 & 0 & -33 & 72
\end{array}\right), q = \left(\begin{array}{r} 1 \\ 1 \\ -1 \\ 1
\end{array}\right).
\end{equation}
This linear complementarity problem was found by a computer search,
with the goal of establishing Lemma~\ref{lem:cycle} below.  It can be
checked that $M$ is a tridiagonal P-matrix, but not a Z-matrix (a
matrix with nonpositive off-diagonal entries). To show that some
other known polynomial-time manageable matrix classes fail to contain
all tridiagonal P-matrices, we need a new concept.

\begin{definition}
Let $\Or(M,q)$ be the digraph with vertex set $2^{[n]}$ and arc set  
\[
\{(B, B\oplus\{i\}): B\subseteq [n], i\in [n],
(\overline{M}_B^{-1}q)_i<0\}.
\] 
\end{definition}
This digraph was first studied by Stickney \&
Watson~\cite{StiWat:Digraph-models}. Under nondegeneracy of
$\lcp(M,q)$, it has a unique sink that coincides with the optimal
basis.

\begin{lem}\label{lem:cycle}
  For $M,q$ as in (\ref{eq:Mq}), $\Or(M,q)$ contains the
  directed cycle $\{1,2,3,4\}\rightarrow\{1,2,3\}
  \rightarrow\{1,2\}\rightarrow\{2\}\rightarrow\emptyset\rightarrow
  \{3\}\rightarrow\{3,4\}\rightarrow\{2,3,4\}\rightarrow\{1,2,3,4\}$.
\end{lem}
We omit the elementary proof. This implies that $M$ cannot be a
hidden Z-matrix, since for such matrices, $\Or(M,q)$ is the acyclic
digraph of some geometric hypercube in $\R^n$, with edges directed by
a linear function~\cite{Manga}. For the same reason, the tridiagonal
P-matrix $M^T$ cannot be the transpose of a hidden K-matrix (a
hidden Z-matrix that is also a P-matrix)~\cite{Pang}. We remark that
Morris has constructed a family of lower Hessenberg matrices
$M\in\R^{n\times n}$ such that $\Or(M,q)$ is highly cyclic for
suitable $q\in\R^n$~\cite{Morris}.

\section{Beyond Hessenberg matrices}
It is natural to ask whether $\lcp(M,q)$ can still be solved in
polynomial time if $M$ is a matrix of fixed bandwidth (number of
nonzero diagonals), or fixed half-bandwidth (number of nonzero
diagonals above or below the main diagonal); see Figure~\ref{fig:band}.
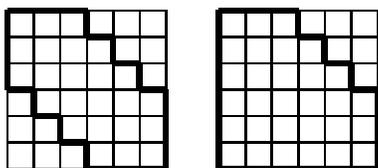
\begin{figure}[htb]
\begin{center}
\begin{picture}(140,60)
% first matrix
% grid
\linethickness{0.5pt}
\put(0,0){\line(1,0){60}}
\put(0,10){\line(1,0){60}}
\put(0,20){\line(1,0){60}}
\put(0,30){\line(1,0){60}}
\put(0,40){\line(1,0){60}}
\put(0,50){\line(1,0){60}}
\put(0,60){\line(1,0){60}}
\put(0,0){\line(0,1){60}}
\put(10,0){\line(0,1){60}}
\put(20,0){\line(0,1){60}}
\put(30,0){\line(0,1){60}}
\put(40,0){\line(0,1){60}}
\put(50,0){\line(0,1){60}}
\put(60,0){\line(0,1){60}}
% nonzero elements
\linethickness{2pt}
% upper chain
\put(0,60){\line(1,0){30}}
\put(30,60){\line(0,-1){10}}
\put(30,50){\line(1,0){10}}
\put(40,50){\line(0,-1){10}}
\put(40,40){\line(1,0){10}}
\put(50,40){\line(0,-1){10}}
\put(50,30){\line(1,0){10}}
\put(60,30){\line(0,-1){30}}
% lower chain
\put(0,60){\line(0,-1){30}}
\put(0,30){\line(1,0){10}}
\put(10,30){\line(0,-1){10}}
\put(10,20){\line(1,0){10}}
\put(20,20){\line(0,-1){10}}
\put(20,10){\line(1,0){10}}
\put(30,10){\line(0,-1){10}}
\put(30,0){\line(1,0){30}}
% second matrix 
% grid
\linethickness{0.5pt}
\put(80,0){\line(1,0){60}}
\put(80,10){\line(1,0){60}}
\put(80,20){\line(1,0){60}}
\put(80,30){\line(1,0){60}}
\put(80,40){\line(1,0){60}}
\put(80,50){\line(1,0){60}}
\put(80,60){\line(1,0){60}}
\put(80,0){\line(0,1){60}}
\put(90,0){\line(0,1){60}}
\put(100,0){\line(0,1){60}}
\put(110,0){\line(0,1){60}}
\put(120,0){\line(0,1){60}}
\put(130,0){\line(0,1){60}}
\put(140,0){\line(0,1){60}}
% nonzero elements
\linethickness{2pt}
% upper chain
\put(80,60){\line(1,0){30}}
\put(110,60){\line(0,-1){10}}
\put(110,50){\line(1,0){10}}
\put(120,50){\line(0,-1){10}}
\put(120,40){\line(1,0){10}}
\put(130,40){\line(0,-1){10}}
\put(130,30){\line(1,0){10}}
\put(140,30){\line(0,-1){30}}
% lower chain
\put(80,0){\line(1,0){60}}
\put(80,0){\line(0,1){60}}
\end{picture}
\end{center}
\caption{Matrices of bandwidth 5 (left) and right half-bandwidth 2
  (right).\label{fig:band}}
\end{figure}

Let $M$ be of fixed right half-bandwidth $t$. Generalizing
Theorem~\ref{thm:dyn}, one can prove that there are only polynomially
many candidates for $B(k)$, \emph{provided that} $B(k)$ has a
\emph{$t$-hole}, meaning that $B(k)$ is disjoint from some contiguous
$t$-element subset of $[k]$.

The only subset of $[k]$ without a $1$-hole is the set $[k]$ itself,
and this is why the lower Hessenberg case $t=1$ is easy. But there is
already an exponential number of subsets of $[k]$ without a $2$-hole.
Hence, the above approach fails for $t\geq 2$. It remains open whether
there is another polynomial-time algorithm in the case of fixed
(right) bandwidth.

\end{document}